\documentclass[12pt]{amsart}

\usepackage{pb-diagram} 

\usepackage{pdfsync}

\usepackage{geometry}
\usepackage{lscape}
\usepackage{graphicx}
\usepackage{epstopdf}    
\usepackage[matrix,arrow,curve,frame,pdf]{xy}
\usepackage{amsmath,amsthm,amssymb,enumerate}
\usepackage{latexsym}
\usepackage{amscd}
\usepackage{amssymb}

\numberwithin{equation}{section}

\theoremstyle{plain}  
\newtheorem{thm}[equation]{Theorem}

\theoremstyle{definition}  

\newtheorem{remark}[equation]{Remark}

\newcommand{\wt}{\widetilde}

\newcommand{\Z}{\mathbb Z}

 \newcommand{\Wedge} {\vee}

\begin{document}

\pagestyle{plain}

\title
{The Morava $K$-theory of $BO(q)$ and $MO(q)$}
\author{Nitu Kitchloo}
\author{W. Stephen Wilson}
\address{Department of Mathematics, Johns Hopkins University, Baltimore, USA}
\email{nitu@math.jhu.edu}
\email{wsw@math.jhu.edu}
\thanks{Nitu Kitchloo is supported in part by NSF through grant DMS
  1307875.}

\date{\today}


{\abstract
We give an easy proof that the Morava $K$-theories for $BO(q)$
and $MO(q)$ are in even degrees.  Although this is a known result,
it had followed from a difficult proof that $BP^*(BO(q))$ was
Landweber flat.  Landweber flatness follows from the even Morava
$K$-theory.  We go further and compute an explicit description of
$K(n)_*(BO(q))$ and $K(n)_*(MO(q))$ and reconcile it with the
purely algebraic construct from Landweber flatness.
}

\maketitle


\section{Introduction}

We are concerned with
the (co)homology theory, Morava $K$-theory, $K(n)^*(-)$,
where $K(n)_* = \Z/2 [v_n^{\pm 1}]$
with the degree of $v_n$ equal to $2(2^n - 1 ) $ (we are
only concerned with $p=2$).

What brought us to the problem of computing the Morava $K$-theories
of the spaces $BO(q)$ was a real need to have $BP^*(BO(q))$ be 
Landweber flat
(in the sense \cite{Land:Hom}) for \cite{Kitch-Wil-BO}.
$BP^*(BO(q))$ had been computed in \cite{WSW:BO} and was shown to
be Landweber flat 
in \cite{KY}, with some seriously complex computations.
Kono and Yagita went on to show that 
$K(n)^*(BO(q))$ was concentrated in even degrees because $BP^*(BO(q))$
was.

The computation in \cite{KY} does not give an explicit answer to
what $K(n)^*(BO(q))$ is, only that it is even degree.  If it is known
that $K(n)^*(BO(q))$ is even degree for all $n$, then the results
of \cite{RWY} show that $BP^*(BO(q))$ is Landweber flat, without having to
compute it.

We present here an easy proof that $K(n)_*(BO(q))$ is even degree and
then
go further and give a basis.
Duality for Morava $K$-theory is straightforward, so $K(n)^*(BO(q))$
is also even degree.

\begin{thm}[\cite{KY}]
\label{even}

\ 

\begin{enumerate}[(i)]
\item
$K(n)^*(BO(q))$ and $K(n)^*(MO(q))$ are even degree for all $n$.
\item
$BP^*(BO(q))$ is Landweber flat.
\end{enumerate}
\end{thm}

As mentioned, $(ii)$ follows directly from $(i)$ using \cite{RWY}
but Kono and Yagita prove $(ii)$ first and then $(i)$.
$(i)$ will be proven in Section 3.

We work with the homology version of the theories and have:

\newpage

\begin{thm}
\label{detail}
\ 
\begin{enumerate}[(i)]
\item
There are elements $b_{2i} \in K(n)_{2i}(BO(1))$ for $0 < i < 2^n$
coming from $K(n)_{2i}(RP^\infty)$.
\item
There are elements $c_{4i} \in K(n)_{4i}(BO(2))$ for $2^n \le i $.
\item
Using products from the standard maps $BO(i)\times BO(j) \rightarrow 
BO(i+j)$, a basis for the reduced homology, $\wt{K(n)}_*(BO(q))$, is:
$$
 \{b_{2i_1}b_{2i_2}\ldots b_{2i_k}
c_{4j_1}c_{4j_2}\ldots c_{4j_m}\}
\quad
0 < k+2m \le q.
$$
$$
0 < i_1 \le i_2 \le \ldots \le i_k < 2^n
\le j_1 \le j_2 \le \ldots \le j_m
$$
\item
$\wt{K(n)}_*(MO(q))$ is as above with $k+2m = q$.

\end{enumerate}
\end{thm}

In \cite{WSW:BO}, it was shown that
\begin{equation}
\label{alg}
BP^*(BO(q)) \simeq BP^*[[c_1,c_2,\ldots,c_q]]/(c_1 - c_1^*, c_2 - c_2^*,
\ldots,c_q - c_q^*),
\end{equation}
where $c_j$ is the Conner-Floyd Chern class and $c_j^*$ is its 
complex conjugate.
In \cite{KY}, Kono and Yagita show that $BP^*(BO(q))$ is Landweber
flat and that
\begin{equation}
K(n)^*(BO(q)) \simeq K(n)^* \widehat{\otimes}_{BP^*} BP^*(BO(q)).
\end{equation}
This shows that the Morava $K$-theory is even degree.
We have computed Morava $K$-theory directly to show it is
even degree, so the results of \cite{RWY} also give us
Landweber flatness for $BP^*(BO(q))$.  Either approach
gives us:
\begin{equation}
\label{cohomology}
K(n)^*(BO(q)) \simeq K(n)^*[[c_1,c_2,\ldots,c_q]]/(c_1 - c_1^*, c_2 - c_2^*,
\ldots,c_q - c_q^*).
\end{equation}

This is a purely algebraic construct that looks nothing
like the answer given in this paper.  In Section 5 we reconcile
it with our direct computation of $K(n)_*(BO(q))$ by finding
a basis for it that is consistent with what we find
for $K(n)_*(BO(q))$.

We review some facts about the standard homology of $BO(q)$
in Section 2 and prove the details of Theorem \ref{detail}
in Section 4.

\section{The standard homology of $BO(q)$ and $MO(q)$}

We begin with some review of basic facts about the homology of
$BO$ and $BO(n)$.  All of our (co)homology will be with $\Z/2$
coefficients.  We start with elements
$$
b_i \in \tilde{H}_i(RP^\infty = BO(1)) \quad i > 0.
$$
We have
$$
\tilde{H}_*(BO(1)) = \Z/2\{b_i:i > 0\}
$$
and maps
$$
BO(1) \rightarrow \cdots \rightarrow BO(q-1) \rightarrow
BO(n) \rightarrow \cdots \rightarrow BO.
$$
The image of the above $b_i$ in $H_*(BO)$ give us the well-known
homology of $BO$ as a polynomial algebra:
$$
H_*(BO) = \Z/2[b_1,b_2,\ldots]. 
$$

We also have the usual maps
\begin{equation}
\label{product}
BO(q) \times BO(k) \longrightarrow BO(q+k).
\end{equation}
For homology we only need 
$$
\prod^q BO(1) \longrightarrow BO(q).
$$
Because $b_i b_j = b_j b_i$, we have elements:
$$
b_{i_1}
b_{i_2}
\cdots
b_{i_k}
\in \tilde{H}_*(BO(q))
\quad
\mbox{for}
\quad
0 < k \le q
\quad
\mbox{and}
\quad
0 < i_1 \le i_2 \le \cdots \le i_k.
$$
These elements form a basis for the reduced homology
of $BO(q)$. 

As an aside, if that is not commonly understood, we
can quickly use the better known cohomology of $BO(q)$ to
see that the size is right.
We have
$$
H^*(BO(q)) = \Z/2[w_1,w_2,\ldots,w_q],
$$
a polynomial algebra on the Stiefel-Whitney classes.
If, by induction, we know $H_*(BO(q-1))$, all we have
to do 
to see the size is right
is show that the elements with $k = q$ above are
in one-to-one correspondence with the ideal generated
by $w_q \in H^*(BO(q))$.  That correspondence is easily
given by
$$
0 < i_1 \le i_2 \le \cdots \le i_q \quad \mbox{ goes to } \quad
w_q^{i_1} w_{q-1}^{i_2 - i_1} w_{q-2}^{i_3-i_2} \cdots w_1^{i_q - i_{q-1}}.
$$

The Steenrod algebra operates on the mod 2 homology of $BO$ and $BO(q)$.
As an element of the Steenrod algebra operates on an element
$
b_{i_1}
b_{i_2}
\cdots
b_{i_k}
$,
it does not alter the number of $b$'s, so we can define:
$$
M_q = \Z/2\{
b_{i_1}
b_{i_2}
\cdots
b_{i_q}
\}
\quad
\mbox{for}
\quad
0 < i_1 \le i_2 \le \cdots \le i_q
$$
and we get the reduced homology
\begin{equation}
\label{equ1}
\tilde{H}_*(BO(q)) = \bigoplus_{j=1}^q M_j
\end{equation}
and
\begin{equation}
\label{equ2}
\tilde{H}_*(BO) = \bigoplus_{j=1}^\infty M_j
\end{equation}
as modules over the Steenrod algebra.

From \cite{MitPrid} we know that stably $BO(q) \simeq \Wedge_{1\le i \le q}
MO(i)$, so, stably, $BO(q) \simeq BO(q-1) \Wedge MO(q)$.  From this we
see that $M_q = H_*(MO(q))$.

\section{The Morava $K$-theories of $BO(q)$ and $MO(q)$ are even}

The first differential in the Atiyah-Hirzebruch spectral sequence (AHSS),
$H_*(X;K(n)_*)$, is just the Milnor primitive, $Q_n$, which is
easy to evaluate in $H_*(BO(1))$ as it just takes $b_{2k}$ to
$b_{2k+1-2^{n+1}}$, as long as $2k > 2^{n+1}-1$.

\begin{remark}

After the first differential, the AHSS collapses for 
$K(n)_*(BO(1))$ because the AHSS is even degree.
The reduced homology is
$K(n)_*$ free on $\{b_2,b_4, \ldots, b_{2^{n+1}-2} \}$.

\end{remark}

\begin{remark}

More interesting is that
after the first differential for $BO$ we are also done, with the
polynomial
result, from the AHSS:
$$
K(n)_*(BO) \simeq K(n)_*[b_2,b_4,\ldots,b_{2^{n+1}-2}]\otimes 
K(n)_*[b_{2i}^2 : i \ge 2^n ],
$$
which was done in \cite{RWY}.  The differential, or as we prefer to say,
the $Q_n$ homology, is computed by pairing up what is missing above
as
$$
P(b_{2i+1})\otimes E(b_{2i+2^{n+1}}).
$$ 
Each of these has trivial $Q_n$ homology.
This collapses after this first differential because it is even degree.
Since $b_{2i}$, $i \ge 2^n$, is not an element, the notation is
misleading.  Later, we will give this generator the name $c_{4i}$.
The element exists in $k(n)_*(BO)$ and reduces to $b_{2i}^2$ in $H_*(BO)$.

\end{remark}

\begin{proof}[Proof of Theorem \ref{even}]
Now we know that the first differential of the AHSS is all it takes
to get $K(n)_*(BO)$ and see that it is all in even degrees. 
The first differential is just
an operation from the Steenrod algebra, $Q_n$.  
By Equation \eqref{equ2}, we must have the $Q_n$ homology of each $M_j$
in even degrees.  From this we see that $K(n)_*(BO(q))$
and $K(n)_*(MO(q))$
must be in even degrees, and by standard Morava $K$-theory duality,
$K(n)^*(BO(q))$ is in even degrees.  This completes the
proof of Theorem \ref{even}.
\end{proof}

\section{The details of the Morava $K$-theories of $BO(q)$ and $MO(q)$ }

All of the homology of $BO(q)$ came from products of elements from
$BO(1)$.  For Morava $K$-theory we have to use elements from $BO(2)$
as well.

Two kinds of elements in $K(n)_*(BO(2))$ come from $K(n)_*(BO(1))$.
First we have the image coming from the map $BO(1) \rightarrow BO(2)$,
i.e.
 $K(n)_*\{b_2,b_4, \ldots, b_{2^{n+1}-2} \}$.
Our second kind comes from the product, $BO(1)\times BO(1) \rightarrow
BO(2)$, which gives:
$$
K(n)_*\{b_{2i_1}b_{2i_2}\} \quad 0 < i_1 \le i_2 < 2^n. 
$$
There are more elements that come from $M_2$ in $K(n)_*(BO(2))$.
In particular, from the computation of $K(n)_*(BO)$ we know that
all $b_{2j}^2$ survive. 
These elements live in $M_2$ so actually survive to
$K(n)_*(BO(2))$.  Consequently, between $K(n)_*(BO(1))$ and
$K(n)_*(BO(2))$, we have all the multiplicative generators
of $K(n)_*(BO)$.  We easily see which $M_q$ these multiple products
live in by the number of $b$'s.

We can now pretty much read off the description of a basis
for $K(n)_*(BO(q))$.  To make the description a little easier
to read, we can consider the part that comes from $M_q$ and
call it $M_q^K = K(n)_*(MO(q))$.  Then we have:
$$
K(n)_*(BO(q)) \simeq K(n)_*(BO(q-1)) \bigoplus K(n)_*(MO(q)).
$$

We are not using the splitting from \cite{MitPrid} to compute
$K(n)_*(BO(q))$, only to compute $K(n)_*(MO(q))$.

Let's give new names to the elements in $K(n)_*(BO(2))$ represented
by $b_{2j}^2$ so we won't have the non-existent product hanging
around.  Let's set $c_{4j} = b_{2j}^2$ for $j \ge 2^n$.

We can now give an explicit description of $M_q^K = K(n)_*(MO(q))$. 
$$
M_q^K \simeq K(n)_*\{b_{2i_1}b_{2i_2}\ldots b_{2i_k}
c_{4j_1}c_{4j_2}\ldots c_{4j_m}\}
\quad
k+2m=q.
$$
$$
0 < i_1 \le i_2 \le \ldots \le i_k < 2^n
\le j_1 \le j_2 \le \ldots \le j_m
$$

This completes the proof of Theorem \ref{detail}.

There is still one bit of unaccounted for structure that we should
mention.  Although $K(n)_*(BO(q))$ is not an algebra, it is a
coalgebra.  The coalgebra structure for the $b$'s comes from 
$BO(1)$, so, for $p < 2^n$, we get
$$
\psi(b_{2p}) = \sum_{i+j=p}  b_{2i} \otimes b_{2j}.
$$
The $c_{4j}$ are written in terms of the $b$'s in the AHSS, so
we also know their coproduct modulo $(v_n)$.
It would just be
$$
\psi(c_{4p}) = \psi(b_{2p}^2) =  \sum_{i+j=p}  b_{2i}^2 \otimes b_{2j}^2
\qquad \mod (v_n).
$$
If $i \ge 2^n$, replace $b_{2i}^2$ with $ c_{4i}$.  Do the same with $j$.
We can work modulo $(v_n)$ because this single differential also
computes $k(n)_*(BO(q))$ where we only have non-negative powers of 
$v_n$.

We know that $K(n)_*(BO) \subset K(n)_*(BU)$.  In \cite{KLW}, there
are elements of $K(n)_*(BU)$ named $z_q$ that are our
$c_{4(2^n+q)}$.  In \cite[Theorem 3.14]{KLW}, the $z_q$ are computed
in terms of $K(n)_*(BU)$ modulo $(v_n^2)$, and their complexity, and
consequently the complexity of the coproduct, shows up here already.
This is to be expected given the complexity of the dual algebra structure
from Equation \eqref{cohomology}.

\section{Reconciliation}

The map $BO(q) \rightarrow BU(q)$ automatically gives a map of the
algebraic construct on the right side of Equation \eqref{alg} to
$BP^*(BO(q))$.  The work of
\cite{WSW:BO}
first involves showing the map is surjective, which is done with the
Adams spectral sequence.  To show injectivity, the algebraic
construct is analyzed.  We can use that analysis here to show
what we want.  We have to establish some notation first.

We have $BP^*(CP^\infty) \simeq BP^*[[x]]$, $x \in BP^2(CP^\infty)$ and
$$
\xymatrix{
BP^*(\prod^q CP^\infty) & \simeq & BP^*[[x_1,x_2,\ldots,x_q]] \\
\cup & & \cup \\
BP^*(BU(q)) & \simeq & BP^*[[c_1,c_2,\ldots,c_q]]
}
$$

The inclusion is given by all of the symmetric functions, which are
generated by the elementary symmetric functions given by the $c_k$.

For $I = (i_1,\ldots,i_q)$, let $x^I = x_1^{i_1} \ldots x_q^{i_q}$. 
Two monomials are equivalent if some permutation of the $x_i$ takes 
one to the other.  Define the symmetric function
$$
s_I = \sum x^I
$$
where the sum goes over all monomials equivalent to $x^I$.  The elementary
symmetric function is $c_k = \sum x_1 \ldots x_k$.

Theorem 1.30 of \cite[page 358]{WSW:BO} computes $c_k^*$ for $BP$ as
$$
c_k^* = (-1)^k c_k + \sum_{i > 0} 
v_i s_{2^i,\underbrace{1,1,\ldots,1}_{k-1}} \quad \mod J^2
$$
where $J=(2,v_1,v_2,\ldots)$. 
We know that the generators of $BP^*(BO(q))$ all map non-trivially
to the cohomology $H^*(BO(q))$.  As a result, we can look at
this relation using 
only the coefficients of $k(n)^* = \Z/2[v_n]$ and consider the
relation modulo $(v_n^2)$. 
Inductively,
the only relation we need is $k=q$. This reduces to
$$
c_q - c_q^* = v_n s_{2^n,\underbrace{1,1,\ldots,1}_{q-1}} \qquad \mod (v_n^2).
$$
Note that for $BU(q)$, our relation is
divisible by $c_q = x_1 \ldots x_q$, i.e. 
$$
s_{2^n,1,1,\ldots,1} = c_q s_{2^n-1}.
$$
Because $K(n)^*(BU(q))\simeq K(n)^* \widehat{\otimes} BP^*(BU(q))$, 
we can be quite sloppy with our powers of $v_n$ because we are
going to invert $v_n$ to get our algebraic description in the end.
The degree of $v_n$ is negative, so the more powers of $v_n$, the
higher the degree of the symmetric function.

The following theorem will reconcile our two different descriptions
of $K(n)^*(BO(q))$.

\begin{thm}
A basis for
$ K(n)^*[[c_1,\ldots,c_q]]/(c_1 - c_1^*, 
\ldots,c_q - c_q^*)$
in terms of symmetric functions is given by
$$
s_{IJ} = \sum x_1^{i_1}\ldots x_m^{i_m} x_{m+1}^{j_1}\ldots x_{m+p}^{j_p}
$$
where $0 < i_1 < \ldots < i_m < 2^n$ and $0 \le j_1 \le \ldots \le j_p$ with
$j_{2i-1} = j_{2i}$.
\end{thm}

\begin{remark}
The definition forces $p$ to be even.  If we drop the $i_m < 2^n$
condition, any $s_K$ can be written in this form.  First, just find all
the pairs of equal exponents and create $J$.  Finding $I$ is easy after
that.
\end{remark}

\begin{remark}
All we do in our proof is reduce arbitrary elements to those
in our theorem.  Because we know $K(n)_*(BO(q))$, we know that
there can be no further reduction, so this is a basis.  This
does reconcile the two descriptions though.
\end{remark}

\begin{proof}
The proof is by double induction.  First, it is by induction on $q$.
This is easy to start with $q=1$ where the result is well known and
straightforward, but worth talking about anyway as it illustrates things
to come in the proof.

The relation in $k(n)^*(BU(1))$ that gives $k(n)^*(BO(1))$ and then 
$K(n)^*(BO(1))$
is just $0 = c_1 - c_1^* = v_n s_{2^n} = v_n x^{2^n}$
modulo $(v_n^2)$.  
The induction is on
the degree of the symmetric function, which in this case is just powers
of $x$.
Inverting $v_n$, we see that $x^{2^n}$ is zero modulo higher powers of $x$.

For any $s_{2^n+k} = x^{2^n+k}$, we have 
$$
0 = s_{2^n} s_k = s_{2^n + k} \mod \text{ higher powers of $x$}.
$$
That is, each $s_{2^n+k}$ is zero modulo higher degree symmetric products.
By induction on the degree of the symmetric product (i.e. induction on  $k$)
we push the relation to higher and higher degrees.  In the topology
on $K(n)^*(BU(1)) \simeq K(n)^*[[x]]$, this converges to 
zero, and so each $s_{2^n+k}$, $k \ge 0$,
is really zero.  We remind the reader 
that our relation isn't really $s_{2^n,1,\ldots,1}=0$ 
modulo higher degree symmetric functions.  The relation has a $v_n$ in
front.  Since our relation really is in $k(n)^*(-)$ because it
comes from $BP^*(-)$, all powers of $v_n$
are positive.  Since we are going to invert $v_n$ in the end in order to get
$K(n)^*(-)$, we can be quite loose with our $v_n$'s.

The same thing will happen in the general, arbitrary $q$, case.  However,
for $q > 1$, there are non-trivial basis elements in high degrees, so
this process doesn't have to go to zero in the limit, but could settle on
a basis element.  Either way, it works for our proof.

From our induction on $q$, we assume the result for $q-1$.  
Stably, $BO(q) \simeq BO(q-1) \Wedge MO(q)$ from \cite{MitPrid} as
well as $BU(q) \simeq BU(q-1) \Wedge MU(q)$.
From
\cite{WSW:BO}, we know that $BP^*(MO(q))$ is the ideal in $BP^*(BO(q))$
generated by $c_q$, and so the same is true for $K(n)^*(BO(q))$.  
Of course, the same is true for $BP^*(MU(q))$, $BP^*(BU(q))$, and
$K(n)^*(BU(q))$.
Consequently,
we can focus our attention on the symmetric functions divisible by
$c_q$ when there are only $q$ variables.

We know that $H^*(BU(q))$ is free on the symmetric functions $s_I$
with $I= ( i_1,\ldots,i_q)$.  If all $i_k > 0$, this is a basis
for $H^*(MU(q))$ and if some are not greater than 0, they are part of
the basis for $H^*(BU(q-1))$.  This splitting is only additive, not
multiplicative.
Because there is no torsion, this is all true for $BP^*(-)$, $k(n)^*(-)$
and $K(n)^*(-)$ as well.

Our next induction is on the degree of the symmetric functions.  We
will show that elements not of the form in our theorem are zero
modulo higher degree elements.  
We know that $K(n)^*(BO(q))$ is $K(n)^*(BU(q))$ modulo the
relations already described and that $K(n)^*(BU(q))$ is just
given by the usual symmetric functions.  
To prove our result, we will not mod
out our relations, but work with $BU(q)$ and just 
describe how the relations accomplish
what we want.
This will suffice for our purposes.
We begin our induction by noticing that all elements in
degrees less than the degree of 
$s_{2^n,1,\ldots,1}
=c_q s_{2^n-1}$
are in our desired basis. The only element 
in the degree of 
$s_{2^n,1,\ldots,1}$
not in the basis
is our relation element, which is zero modulo higher degree symmetric
functions (ignoring the $v_n$ as discussed above).

An arbitrary element not of the form in the theorem simply has $i_m \ge 2^n$
instead of $i_m < 2^n$.  Having fixed a degree, we first consider the
cases where $i_m = 2^n + k$, with $k > 0$.  Since we are working with
elements divisible by $c_q$, we can divide by $c_q$ to get a new symmetric
function, $s_{I'J'}$, with each $i_s$ replaced by $i_s - 1$ and the same for the
$j_s$.  This symmetric function has $i_m' = 2^n +k -1$.  Since $k > 0$,
this is known to be zero modulo higher degrees by our induction on
degree.  
Multiplying by
$c_k$ to get our original symmetric function, we see it must be zero
modulo higher degrees.
Note that we are using our induction on $q$ here.  If $i_1$ or $j_1$ 
(or both) are equal to $1$, then $s_{I'J'}$ is in $K(n)^*(BU(q-1))$
because it is not divisible by $c_q$.  By our induction, we know the
behavior of the relations here.

In our fixed degree, we have eliminated all of the bad elements
except those with $i_m = 2^n$.  From such a symmetric function
$s_{IJ}$, we create a new symmetric function $s_{I'J'}$ by
eliminating the $x_m^{i_m}=x_m^{2^n}$ term and subtracting $1$ from 
from all of the other $i_s$ and $j_s$.  We want to analyze
$$
s_{2^n,1,1,\ldots,1}s_{I'J'}.
$$
Since 
$s_{2^n,1,1,\ldots,1}$ 
is zero modulo higher degrees, this
product is too.  Multiplying symmetric functions can be tricky because
the result can be a sum of symmetric functions.
The easy one to deal with is when $i_1$ and $j_1$ are greater than one
(recall that $m+p=q$). 
In this case, if your $x^{2^n}$ term is multiplied by any power of $x$,
we are in the situation where our product has $x^{2^n + k}$, with $k >0$,
and we have dealt with those terms already.  The only thing left
is to multiply the $x^{2^n}$ back into the place it was removed from
and then all of the other exponents are raised by 1, giving us
back our original $s_{IJ}$.

Things are slightly more complicated if $i_1$ or $j_1$ is 1.  
(They must be at least 1 because everything is divisible by $c_q$.)
Again, if our $x^{2^n}$ is multiplied by a non-zero power of $x$, 
we get $x^{2^n+k}$ and these terms have been handled already.
Our $x^{2^n}$ must hit an $x^0$ term, but by the definition of
symmetric functions, these are all equivalent, so the other $x_i$ all
just have their exponent raised by 1 in our product 
and we get our $s_{IJ}$ back,
showing it is zero modulo higher degrees.
\end{proof}


\end{document}